\documentclass[10pt, a4paper]{amsart}

\usepackage[T1]{fontenc}
\usepackage[utf8]{inputenc}
\usepackage{amsmath, amssymb, amsfonts, amsthm}
\usepackage[margin=2.0cm]{geometry}
\usepackage[]{setspace}
\usepackage[usenames,dvipsnames]{color}
\usepackage[unicode]{hyperref}
\hypersetup{
    urlcolor=blue,
    colorlinks=true,
}
\usepackage{comment}

\usepackage[all,matrix,arrow]{xy}
\usepackage[shortlabels]{enumitem}
\usepackage{xcolor}

\theoremstyle{plain}
\newtheorem{theorem}{Theorem}[section]
\newtheorem{proposition}[theorem]{Proposition}
\newtheorem{lemma}[theorem]{Lemma}
\newtheorem{corollary}[theorem]{Corollary}

\theoremstyle{definition}

\newtheorem{definition}[theorem]{Definition}

\theoremstyle{remark}
\newtheorem{remark}[theorem]{Remark}

\newcommand\rmDelta{\mathrm{\Delta}}

\newcommand\FF{\mathbb{F}}

\renewcommand\int{\mathbf{int}}

\numberwithin{equation}{section}

\newcommand{\Lring}{\mathcal{L}_{\rm ring}}
\newcommand{\Lval}{\mathcal{L}_{\rm val}}

\title{Axiomatizing the existential theory of $\mathbb{F}_{q}(\!(t)\!)$}

\author{Sylvy Anscombe, Philip Dittmann and Arno Fehm}
\thanks{\today}
\thanks{MSC classes: 	03C60, 12L05, 11D88}
\address{Universit\'{e} Paris Cit\'{e} and Sorbonne Universit\'{e}, CNRS, IMJ-PRG, F-75013 Paris, France}
\email{sylvy.anscombe@imj-prg.fr}
\address{Institut f\"{u}r Algebra, Technische Universit\"{a}t Dresden, 01062 Dresden, Germany}
\email{philip.dittmann@tu-dresden.de}
\address{Institut f\"{u}r Algebra, Technische Universit\"{a}t Dresden, 01062 Dresden, Germany}
\email{arno.fehm@tu-dresden.de}

\begin{document}
 \begin{abstract}
 We study the existential theory of equicharacteristic henselian valued fields with a distinguished uniformizer.
 In particular, assuming a weak consequence of resolution of singularities, we obtain an
 axiomatization of -- and therefore an algorithm to decide --
 the existential theory relative to the existential theory of the residue field.
 This is both more general and works under weaker resolution hypotheses than the algorithm of Denef and Schoutens,
 which we also discuss in detail.
 In fact, the consequence of resolution of singularities our results are conditional on is the weakest under which they hold true.
 \end{abstract}

\maketitle

\section{Introduction}

\noindent
Hilbert's tenth problem asks for an algorithm to decide whether a given polynomial over $\mathbb{Z}$ has a solution in $\mathbb{Z}$, which was shown to be impossible by work of Davis, Putnam, Robinson and Matiyasevich.
Analogues for various other rings and fields have been studied since, see \cite{Poonen_survey,Shlapentokh,Koenigsmann_survey}.
For local fields of positive characteristic, i.e.~Laurent series fields $\mathbb{F}_q(\!(t)\!)$ over a finite field with $q=p^n$ elements, there are two results in the literature: 
The first one works with polynomials over $\mathbb{F}_p(t)$, which is arguably the correct analogue of $\mathbb{Z}$ in this setting, but needs to assume the truth of a deep and unresolved conjecture from algebraic geometry, resolution of singularities (see Section \ref{sec:resolution} for precise statements and discussion):

\begin{theorem}[Denef--Schoutens {\cite[Theorem 4.3]{DS}}]\label{thm:DS1}
Assume that resolution of singularities holds
in characteristic $p$.
There exists an algorithm that decides whether a given system of polynomial equations over $\mathbb{F}_p(t)$
has a solution in $\mathbb{F}_q(\!(t)\!)$.
\end{theorem}

The second one, which is proved by completely different methods, needs no such hypothesis but allows for polynomials only over $\mathbb{F}_p$:

\begin{theorem}[Anscombe--Fehm {\cite[Corollary 7.7]{AF}}]\label{thm:AF1}
There exists an algorithm that decides whether a given system of polynomial equations over $\mathbb{F}_p$
has a solution in $\mathbb{F}_q(\!(t)\!)$.
\end{theorem}

Both results have recently found various applications, see for example 
\cite{Onay,MRUV, Kostas1,Kostas2,Kostas3}.
In fact, both results are more general than stated here, and these more general results are most naturally phrased in the language of the model theory of valued fields:

\begin{definition}
We denote by $\mathcal{L}_{\rm ring}=\{+,-,\cdot,0,1\}$ the language of rings,
by $\mathcal{L}_{\rm val}=\mathcal{L}_{\rm ring}\cup\{\mathcal{O}\}$ the language of valued fields,
where $\mathcal{O}$ is a unary predicate symbol,
and by $\mathcal{L}_{\rm ring}(\varpi)$ and
$\mathcal{L}_{\rm val}(\varpi)$ the respective language expanded by a constant symbol $\varpi$.
We view a valued field $(K,v)$ as an $\mathcal{L}_{\rm val}$-structure by interpreting $\mathcal{O}$ as the valuation ring $\mathcal{O}_v$ of $v$.
A {\em uniformizer} of a valued field $(K,v)$ is an element $\pi\in K$ of smallest positive value. 
Given a distinguished uniformizer $\pi$ of $(K,v)$, we view
$(K,v,\pi)$ as an $\mathcal{L}_{\rm val}(\varpi)$-structure by interpreting $\varpi$ as $\pi$.
\end{definition}

So for example $\mathbb{F}_q(\!(t)\!)$ carries the natural equicharacteristic henselian valuation $v_t$ (the $t$-adic valuation), and if we view
$(\mathbb{F}_q(\!(t)\!),v_t,t)$ as an $\mathcal{L}_{\rm val}(\varpi)$-structure, the fact that a given system of polynomial equations over $\mathbb{F}_p(t)$ has a solution in $\mathbb{F}_q(\!(t)\!)$ can be expressed as the truth in $\mathbb{F}_q(\!(t)\!)$ of a particular existential $\mathcal{L}_{\rm ring}(\varpi)$-sentence
(by existential we mean of the form $\exists x_1\dots\exists x_n\psi(\underline{x})$
with a quantifier-free formula $\psi(\underline{x})$).
The general result of Denef and Schoutens,
which we revisit and discuss in Section \ref{sec:DS},
can now be phrased as follows:

\begin{theorem}[Denef--Schoutens {\cite[Theorem 4.3]{DS}}]\label{thm:DS2}
Let $(K,v)$ be an equicharacteristic henselian valued field with distinguished uniformizer $\pi$.
Assume that
\begin{enumerate}
    \item\label{cond:DS1} resolution of singularities holds
    in characteristic $p$,
    \item\label{cond:DS2}  $\mathcal{O}_v$
is a discrete valuation ring, and
   \item\label{cond:DS3} $\mathcal{O}_v$ is excellent.
\end{enumerate}
Then the existential $\mathcal{L}_{\rm val}(\varpi)$-theory ${\rm Th}_\exists(K,v,\pi)$ of $(K,v,\pi)$ is decidable
relative to the existential $\mathcal{L}_{\rm ring}$-theory ${\rm Th}_\exists(Kv)$ of the residue field $Kv$
(in the sense of Turing reduction, cf.~Remark \ref{rem:relatively-decidable_0}).
\end{theorem}

The more general form of Theorem \ref{thm:AF1}
({\cite[Theorem 6.5]{AF}})
gives an axiomatization of the existential theory of an equicharacteristic henselian valued field relative to the existential theory of the residue field, 
which in particular also implies a relative decidability result but is interesting in its own right, 
and useful for example when questions of uniformity are concerned.

The goal of this work is to show that the approach from 
\cite{AF} can be extended to the setting of Theorem \ref{thm:DS2} as well,
thereby combining the best of both worlds.
Besides also obtaining an axiomatization of the existential theory,
the assumption (\ref{cond:DS1}) can thereby be weakened to a local version of it, called local uniformization
(see Section \ref{sec:resolution} for definitions and an extensive discussion), and assumptions (\ref{cond:DS2}) and (\ref{cond:DS3}) can be completely eliminated.
In Section \ref{sec:axiomatization} on p.~\pageref{proof_of_15} we prove:

\begin{theorem}\label{thm:intro_main}
Assume that
consequence \eqref{R:ExClosed} 
of local uniformization holds (p.~\pageref{R:ExClosed}).
Let $(K,v)$ be an equicharacteristic henselian valued field with distinguished uniformizer $\pi$.
Then the universal/existential $\mathcal{L}_{\rm val}(\varpi)$-theory of $(K,v,\pi)$ is entailed by
\begin{enumerate}[(i)]
    \item  $\mathcal{L}_{\rm val}$-axioms for equicharacteristic henselian valued fields,
    \item the $\mathcal{L}_{\rm val}(\varpi)$-axiom expressing that $\pi$ has smallest positive value, and
    \item $\mathcal{L}_{\rm val}$-axioms expressing that the residue field models the universal/existential $\mathcal{L}_{\rm ring}$-theory of $Kv$.
\end{enumerate} 
In particular, the existential $\mathcal{L}_{\rm val}(\varpi)$-theory ${\rm Th}_\exists(K,v,\pi)$ of $(K,v,\pi)$
is decidable relative to the existential $\mathcal{L}_{\rm ring}$-theory ${\rm Th}_\exists(Kv)$ of $Kv$.
\end{theorem}

Here, by universal/existential we mean  sentences that are either universal or existential.
In fact, in Theorem~\ref{thm:main} we prove this with parameters
from a base field more general than $\mathbb{F}_p(\pi)$.
In this strong form, this axiomatization statement for universal/existential theories is in fact equivalent to \eqref{R:ExClosed}, see Remark \ref{rem:R4}, so our hypothesis cannot be weakened further.

Our results in particular give a new proof of Theorem \ref{thm:DS1}
which works under this weaker hypothesis:

\begin{corollary}\label{cor:intro_Fpt}
Assume that
consequence \eqref{R:ExClosed}
of local uniformization holds.
There exists an algorithm that decides whether a given system of polynomial equations over $\mathbb{F}_q(t)$
has a solution in $\mathbb{F}_q(\!(t)\!)$.
\end{corollary}

Recent work of Kartas includes related results,
see Remark~\ref{rem:Kostas}.

\section{Resolution of singularities and local uniformization}
\label{sec:resolution}

\noindent
In this section we discuss several versions of resolution of singularities and some of its consequences.

\begin{definition}
  Let $K$ be a field.
  A \emph{$K$-variety} is an integral separated $K$-scheme of finite type,
  and a morphism of $K$-varieties $Y \to X$ is simply a morphism of $K$-schemes.

  Given a $K$-variety $X$, a \emph{resolution of singularities} (or for short \emph{resolution}) of $X$ is a proper birational morphism $Y \to X$ with $Y$ a regular $K$-variety.
  A \emph{blowing-up resolution} is a morphism $Y \to X$ arising as the blowing-up along some closed proper subscheme $Z \subset X$ such that $Y$ is regular.

  By a valuation $v$ on a field extension $F/K$ we mean a valuation
  on $F$ which is trivial on $K$.
  Given a finitely generated field extension $F/K$ and a valuation $v$ on $F/K$, a \emph{local uniformization} of $v$ consists of a $K$-variety $Y$ and an isomorphism $F \cong_K K(Y)$ such that, under the identification of $F$ with $K(Y)$, $v$ is centred on a regular point $y$ of $Y$ (i.e.\ $\mathcal{O}_v$ dominates the regular local ring $\mathcal{O}_{Y,y}$).
\end{definition}

\begin{lemma}\label{lem:resolutionConsequences}
  Let $K$ be a field.
  \begin{enumerate}
  \item Every blowing-up resolution of a $K$-variety $X$ is a resolution of $X$.
  \item If a proper $K$-variety $X$ has a resolution, then any valuation $v$ on $K(X)/K$ has a local uniformization.
  \item Let $F/K$ be a finitely generated field extension and $v$ a valuation on $F/K$ which has residue field $Fv = K$.
    If $v$ has a local uniformization, then there exists a $K$-embedding of $F$ into $K(\!(t)\!)$.
  \end{enumerate}
\end{lemma}
\begin{proof}
  Every blowing-up morphism along a closed proper subscheme is proper and birational \cite[Proposition 8.1.12~(d)]{Liu}, implying (1).
  For (2), let $Y \to X$ be a resolution of $X$.
  Every valuation $v$ on the function field $K(Y)/K$ is centred at some point of $Y$ by the valuative criterion of properness, cf.\ the discussion in \cite[Definition 8.3.17]{Liu}.
  In particular, every such $v$ has a local uniformization, and the same holds for valuations on $K(X) \cong_K K(Y)$.
  In the situation of (3), the field extension $F/K$ is necessarily separable \cite[Lemma 2.6.9]{FJ},
  so the statement follows from \cite[Lemma 9]{Fehm2011};
  alternatively, it can be deduced from
  \cite[Theorem~13]{Kuhlmann_Places}.
\end{proof}

The condition that a finitely generated field extension $F/K$ can be embedded over $K$ into $K(\!(t)\!)$ is especially interesting when $K$ is a \emph{large} (or \emph{ample}) field
in the sense of \cite{Pop}, 
i.e.~$K$ is existentially closed in $K(\!(t)\!)$ in $\mathcal{L}_{\rm ring}$.
Apart from the definition the only important facts for us are
that large fields form an elementary class, cf.~\cite[Remark 1.3]{Pop},
and Ershov's result \cite[Proposition 2.1]{Ershov} that henselian nontrivially valued fields are large.
See \cite[Example 5.6.2]{Jardenpatching} and \cite{Pop2} for more modern accounts
that prove this (using an equivalent definition of largeness),
and see \cite{BFsurvey} for more background on large fields.

We are led to consider the following hypotheses:
\begin{enumerate}[label=(R\arabic*), ref=R\arabic*]
\setcounter{enumi}{-1}
\item\label{R:Blow} For every field $K$ and every $K$-variety $X$ there exists a blowing-up resolution $Y \to X$.
\item\label{R:Mod} For every field $K$ and every $K$-variety $X$ there exists a resolution $Y \to X$.
\item\label{R:LocUnif} For every field $K$, every finitely generated field extension $F/K$ and every valuation $v$ on $F/K$ there exists a local uniformization of $v$.
\item\label{R:Zvalued} For every field $K$ and every nontrivial finitely generated extension $F/K$ such that there exists a valuation $v$ on $F/K$ with residue field $Fv = K$, there also exists a valuation with value group $\mathbb Z$ which has that property.
\item\label{R:ExClosed} Every large field $K$ is existentially closed in every extension $F/K$ for which there exists a valuation $v$ on $F/K$ with residue field $Fv = K$. 
\end{enumerate}

\begin{proposition}
The following implications hold true:
$\eqref{R:Blow} \Rightarrow \eqref{R:Mod} \Rightarrow \eqref{R:LocUnif} \Rightarrow \eqref{R:Zvalued} \Leftrightarrow \eqref{R:ExClosed}$
\end{proposition}

\begin{proof}
The implication $\eqref{R:Blow} \Rightarrow \eqref{R:Mod}$ follow from
Lemma \ref{lem:resolutionConsequences}(1).
For $\eqref{R:Mod} \Rightarrow \eqref{R:LocUnif}$
first choose any proper $K$-variety $X$ with $K(X)\cong_K F$,
and then apply Lemma \ref{lem:resolutionConsequences}(2).
By Lemma \ref{lem:resolutionConsequences}(3), $\eqref{R:LocUnif}$
implies that every $F/K$ as in $\eqref{R:Zvalued}$ can be embedded into $K(\!(t)\!)$,
and the restriction of the $t$-adic valuation to $F$ then satisfies $Fv=K$ and $vF \cong \mathbb{Z}$, which shows that $\eqref{R:LocUnif}\Rightarrow\eqref{R:Zvalued}$.
If $F/K$ is as in $\eqref{R:ExClosed}$, then
$\eqref{R:Zvalued}$ gives a valuation $v$ on $F/K$ with $Fv=K$ and $vF=\mathbb{Z}$.
The completion $(\hat{F},\hat{v})$ of $(F,v)$ is then isomorphic over $K$ to $(K(\!(t)\!),v_t)$, cf.~Proposition~\ref{prop:complete_structure} below,
so the large field $K$ is existentially closed in $\hat{F}$ and therefore in particular in $F$,
which shows that $\eqref{R:Zvalued} \Rightarrow \eqref{R:ExClosed}$.
Finally, we have that $\eqref{R:ExClosed} \Rightarrow \eqref{R:Zvalued}$ by \cite[Corollary~10]{Fehm2011}.
\end{proof}

The resolution of singularities that Denef and Schoutens assume
is $\eqref{R:Blow}$, see Section \ref{sec:DS}.
In our main theorem (Theorem \ref{thm:main}), we will work under the weaker condition \eqref{R:ExClosed}.

\begin{remark}\label{rem:Rn}
  The terminology on resolution of singularities is not entirely standardized.
  Depending on the author, either a resolution or a blowing-up resolution in our terminology might be called a resolution, a desingularization or a weak desingularization.
  For a resolution $Y \to X$, it may additionally be demanded that it is an isomorphism over the regular locus of $X$.
  For varieties $X$ over a field $K$ of characteristic zero, the existence of a blowing-up resolution in this strong sense was established in \cite{Hironaka}.
  Therefore all five hypotheses above are valid when restricted to fields $K$ of characteristic zero.

  In general, it is known that every variety of dimension three or less has a resolution of singularities (see \cite{CossartPiltant_ResolutionArithmetical} and the literature referenced there), and it follows that \eqref{R:LocUnif}, \eqref{R:Zvalued} and \eqref{R:ExClosed} hold restricted to field extensions $F/K$ of transcendence degree at most three.

  Local uniformization is occasionally stated in a stronger version (see for instance the introduction to \cite{Temkin_InsepLocalUnif}), where one demands that for any given $K$-variety $X$ with function field $K(X) \cong_K F$ on which $v$ is centred, the uniformizing $Y$ can be chosen to dominate $X$, i.e.\ to come with a birational morphism $Y \to X$ compatible with the isomorphism $K(X) \cong F \cong K(Y)$.
  The best result known here seems to be that \eqref{R:LocUnif} holds up to replacing $F$ by a purely inseparable finite extension, see \cite[Section 1.3]{Temkin_InsepLocalUnif}.
  It follows that \eqref{R:Zvalued} and \eqref{R:ExClosed} hold when restricted to perfect base fields $K$.
  In fact, this restricted form of \eqref{R:ExClosed} was already proven earlier in \cite[Theorem 17]{Kuhlmann_Places}.
  Less deep is the fact 
  (immediate from~\cite[Lemma 2.6.9(b), Proposition 11.3.5]{FJ} and \cite[Lemma 9]{Fehm2011})
  that \eqref{R:Zvalued} and \eqref{R:ExClosed}
  also hold when restricted to {\em pseudo-algebraically closed} base fields $K$,
  i.e.~fields $K$ for which every geometrically integral $K$-variety has a $K$-rational point.
\end{remark}

\begin{remark}
  Desingularization problems are also studied for more general schemes, rather than only varieties over fields.
  For instance, it is conjectured that for any reduced quasi-excellent scheme $X$ there exists a proper birational morphism $Y \to X$ with $Y$ regular \cite[Remarque (7.9.6)]{EGA4_2}.
  This is a generalization of \eqref{R:Mod} since any variety over a field is quasi-excellent.
  Similarly, local uniformization is also studied in this more general setting, see \cite{Temkin_InsepLocalUnif}.
  The condition of excellence also appears in our own results below, but for reasons unrelated to resolution of singularities -- indeed, we only need resolution in the shape of the weak condition \eqref{R:ExClosed}.
\end{remark}

\begin{remark}\label{rem:Kostas}
  Even more variations of resolutions are possible.
  These include the desingularization of a pair of a reduced quasi-excellent scheme $X$ and a nowhere dense closed subscheme $Z$, see \cite[Introduction]{Temkin_DesingCharacZero}.

  A resolution hypothesis of this type is used in \cite{Kostas_tame} to conditionally obtain an axiomatization for existential theories of certain, not necessarily discrete, valuation rings.
  The approach there is based on the geometric method of Denef--Schoutens, which we discuss in Section \ref{sec:DS}.

  The first available preprint of \cite{Kostas_tame} in Section 5 also announces (with sketch proofs) results of this kind conditional only on variants of \eqref{R:LocUnif}.
  These do not appear in the published version
  but are included in \cite{Kostas_thesis}.
\end{remark}

\section{The Denef--Schoutens algorithm}
\label{sec:DS}

\noindent
We now discuss the algorithm given by Denef and Schoutens, with a view towards improving the statements given in \cite{DS}, addressing some subtle points of the algorithm, and avoiding scheme-theoretic language at least to some extent.
Our own technique is entirely different, so the reader interested in new results can skip ahead to Section \ref{sec:axiomatization}.

Throughout,
$R$ is a fixed henselian discrete valuation ring with uniformizer $\pi$ and fraction field $K$, $p > 0$ the characteristic of $K$,
and $\kappa = R/(\pi)$ the residue field of $R$.
Note that $\pi$ is necessarily transcendental
over the prime field $\mathbb{F}_p$ of $K$.
We view $R$ as an $\mathcal{L}_{\rm ring}(\varpi)$-structure $(R,\pi)$ by interpreting $\varpi$ as $\pi$.

\subsection*{Overview}
Theorem \ref{thm:DS2}, concerning the decidability of the existential $\Lval(\varpi)$-theory of $(K,v,\pi)$, is equivalent to the decidability of the existential $\Lring(\varpi)$-theory of $(R, \pi)$ under the hypothesis \eqref{R:Blow},
and the latter is what Proposition \ref{prp:DS} below proves.
To decide the truth of existential $\Lring(\varpi)$-sentences in $R$, one proceeds in three steps:
Deciding the truth of a {\em positive} existential $\Lring(\varpi)$-sentence in $R$ means determining whether a system of polynomial equations has a solution in $R$.
By a variant of a theorem of Greenberg, such a solution exists if and only if there is a solution in the residue ring $R/(\pi^N)$ to the appropriately reduced system, for a sufficiently high (computable) power $\pi^N$.
This reduces the original question about satisfaction of positive existential $\Lring(\varpi)$-sentences in $R$ to a question about the existential $\Lring$-theory of the residue field $\kappa$.

For the second step, one considers existential $\Lring(\varpi)$-sentences (that is, systems of equations and inequations), where the underlying system of equations describes a \emph{regular} affine variety.
A version of the implicit function theorem then shows that any solution in $R$ to the system of equation gives rise to a wealth of other solutions, by perturbing some coordinates.
This then implies that if there is any solution to the equations (which is a computable condition by the first step), there will also be one satisfying the inequations.

In the third step, one considers systems of equations and inequations in full generality; only here does resolution of singularities enter.
Assuming that the system of equations (or rather, the affine scheme described thereby) has a resolution of singularities, one can in fact effectively find such a resolution, simply by enumerating candidates.
In an inductive procedure, one can then relate the existence of solutions in $R$ to the original system of equations and inequations, and solutions to related systems without singularities, which can be handled by the second step.

\subsection*{Step 1: Positive existential $\Lring(\varpi)$-sentences in $R$}

Here we slightly rephrase the results of \cite[Section 3]{DS}.

\begin{lemma}\label{lem:Greenberg}
  Let $f_1, \dotsc, f_n \in \FF_p[t][X_1, \dotsc, X_m]$ be polynomials over the ring $\FF_p[t]$.
  The following are equivalent:
  \begin{enumerate}
  \item The system of equations $\bigwedge_{i = 1}^n f_i(\underline X) = 0$ has a solution in $\kappa[\![t]\!]$.
  \item The system of equations $\bigwedge_{i = 1}^n f_i(\underline X) = 0$ has a solution in $\kappa[t]^h$, the henselization of $\kappa[t]$ at the ideal $(t)$ (i.e.~the valuation ring of the henselization $\kappa(t)^h$ of $\kappa(t)$ with respect to the $t$-adic valuation).
  \item For every $N > 0$, there is a solution in $\kappa[t]/(t^N)$ to the system of equations $\bigwedge_{i = 1}^n f_i(\underline X) = 0$ (more precisely, to the reduction mod $t^N$).
  \end{enumerate}
  Moreover, it suffices to check (3) for one specific large $N$, which can be computed from $n, m$ and the total degree of the $f_i$ (in the variables $t, X_1, \dotsc, X_m$).
\end{lemma}
\begin{proof}
  It is clear that $(2)\Rightarrow (1)$ (since $\kappa[t]^h \subseteq \kappa[\![t]\!]$) and $(1)\Rightarrow(3)$ (by reduction modulo $t^N$).
  The implication $(3)\Rightarrow(2)$ is Greenberg's result \cite[Corollary 2]{Greenberg}.
  The reduction to one specific large $N$ comes down to an \emph{effective} version of Greenberg's theorem, which is a special case of the effective Artin approximation given in \cite[Theorem 6.1]{UltraproductsApproximationI} (take $n=1$ and $\alpha = 0$ in the notation there).
\end{proof}

\begin{lemma}\label{lem:reduce_to_residue_field}  
Given finitely many polynomials $f_1, \dotsc, f_n \in \FF_p[T,X_1, \dotsc, X_m]$, one can effectively find finitely many polynomials $g_1, \dotsc, g_{n'} \in \FF_p[Y_1, \dotsc, Y_{m'}]$ such that the system of equations $\bigwedge_{i=1}^n f_i(\pi,\underline X) = 0$ has a solution in $R$ if and only if the system of equations $\bigwedge_{j=1}^{n'} g_j(\underline Y) = 0$ has a solution in $\kappa$.
  In particular, the positive existential $\Lring(\varpi)$-theory of $(R,\pi)$ is decidable relative to the existential $\Lring$-theory of $\kappa$.
\end{lemma}

\begin{proof}
The ring $R$ embeds into its completion $\widehat{R}$.
By the Cohen structure theorem
(see also Proposition \ref{prop:complete_structure} and Proposition \ref{prop:complete_lift} below)
there is an isomorphism
$\varphi \colon \widehat{R}\rightarrow\kappa[\![t]\!]$
that sends $\pi$ to $t$.
For each $N>0$, $\varphi$ induces an isomorphism 
$$
\varphi_{N} \colon R/(\pi^N)\cong\hat{R}/(\pi^N)\rightarrow\kappa[\![t]\!]/(t^N) \cong \kappa[t]/(t^N)
$$
that maps $\pi+(\pi^{N})$ to $t+(t^{N})$.

Let $N > 0$ be the specific $N$ given in the last point of Lemma \ref{lem:Greenberg},
computed from the polynomials
$f_i(t,\underline{X})\in\mathbb{F}_{p}[t][X_{1},\ldots,X_{m}]$.
We claim that the system of equations
$(\textsc{i})\colon \bigwedge_{i=1}^n f_i(\pi,\underline X) = 0$
has a solution in $R$
if and only if
the reduction of
$(\textsc{ii})\colon \bigwedge_{i=1}^n f_i(t,\underline X) = 0$
has a solution in
$\kappa[t]/(t^{N})$.

One direction is trivial: if there is a solution to
(\textsc{i})
in $R$ then there is a solution
to the reduction of
(\textsc{i})
in $R/(\pi^{N})$,
and via the isomorphism $\varphi_{N}$
there is a solution to the reduction of 
(\textsc{ii})
in
$\kappa[t]/(t^{N})$.
Conversely, 
suppose a solution
to the reduction of
(\textsc{ii})
exists in
$\kappa[t]/(t^{N})$.
By our choice of $N$, 
there is a solution to
(\textsc{ii})
in $\kappa[t]^{h}$.
Since $\kappa[t]^h$ is the union of $\kappa_0[t]^h$ where $\kappa_0$ ranges over the finitely generated subfields of $\kappa$,
we have a solution
to
(\textsc{ii})
in $\kappa_0[t]^h$ for some $\kappa_0$.
By \cite[Lemma 2.3]{AF},
there exists a partial section $\kappa_0 \rightarrow R$ of the residue map  $R\rightarrow\kappa$
(see Definition~\ref{def:partial_section}).
This extends to an embedding $\kappa_{0}[t]^{h}\rightarrow R$, sending $t$ to $\pi$.
Thus we have a solution to
(\textsc{i})
in $R$, proving the claim.

The condition that the reduction of
(\textsc{ii})
has a solution in $\kappa[t]/(t^N)$
is easily translated into
the solvability
in $\kappa$
of a system of 
finitely many polynomial equations over
$\mathbb{F}_{p}$
(the $g_j$ of the statement) by a standard interpretation (or Weil restriction) argument, since $\kappa[t]/(t^N)$ is an algebra of finite dimension over $\kappa$.
\end{proof}

\begin{remark}\label{rem:relatively-decidable_0}
  Let us briefly describe the formal meaning of ``one can effectively find'' and ``relatively decidable'' in the statement of Lemma~\ref{lem:reduce_to_residue_field}.

  Define an injection of the set of $\Lring(\varpi)$-terms and $\Lring(\varpi)$-formulas into $\mathbb{N}$
  by a standard Gödel coding.
  Every $\Lring(\varpi)$-term with free variables $X_1, \dotsc, X_n$ induces a polynomial in $\mathbb{F}_p[T, X_1, \dotsc, X_n]$ in the natural way, where the indeterminate $T$ takes the place of the constant symbol $\varpi$.
  We now obtain a Gödel numbering of $\mathbb{F}_p[T, X_1, \dotsc, X_n]$ by assigning to every polynomial the minimal Gödel number of an $\Lring(\varpi)$-term which induces it.
  The set of Gödel numbers of polynomials in $\mathbb{F}_p[T, X_1, \dotsc, X_n]$ is a decidable set of natural numbers, and
  mapping the Gödel number of two polynomials in $\mathbb{F}_p[T, X_1, \dotsc, X_n]$ to the Gödel number of their sum (respectively, their product) is a computable function.
  (This is essentially equivalent to asserting that the numbering gives an explicit representation of the ring $\mathbb{F}_p[T, X_1, \dotsc, X_n]$ in the sense of \cite[§2]{EffectiveProcedures}.)

  The first part of Lemma~\ref{lem:reduce_to_residue_field} now means that there exists a Turing machine which takes as input the codes of finitely many polynomials $f_1, \dotsc, f_n$, and produces as output the codes of finitely many polynomials $g_1, \dotsc, g_{n'}$ with the given property.
  The second part of the lemma asserts that there exists a Turing machine that takes as input
the code
of a positive existential $\mathcal{L}_{\rm ring}(\varpi)$-sentence $\varphi$
and outputs yes or no according to whether
$(R,\pi)\models\varphi$,
using an oracle 
(as described in \cite[Chapter 4]{Shoenfield})
which for the code of 
an existential $\mathcal{L}_{\rm ring}$-sentence $\psi$ decides whether $\kappa\models\psi$.
  In other words, the set of codes of positive existential $\Lring(\varpi)$-sentences satisfied by $(R, \pi)$ is Turing reducible to the set of codes of sentences in $\mathrm{Th}_\exists(\kappa)$.
\end{remark}

\begin{remark}
Although in this section we have fixed a ring $R$, we note that the
algorithm of Lemma~\ref{lem:reduce_to_residue_field} to produce the polynomials $g_{j}$ from the $f_{i}$ does not depend on $R$.
This is because the interpretation (or Weil restriction) argument used is uniform in $\kappa$,
since
$\kappa[t]/(t^{N})$
is an $N$-dimensional $\kappa$-algebra 
with multiplication defined by structure constants in $\mathbb{F}_{p}$
that do not themselves depend on $\kappa$.
In particular,
$\kappa[t]/(t^{N})=\mathbb{F}_{p}[t]/(t^{N})\otimes_{\mathbb{F}_{p}}\kappa$
is an extension of scalars.
\end{remark}

\begin{remark}
  As remarked in \cite[Section 6]{UltraproductsApproximationI}, the effective version of Greenberg's theorem used in the reduction in Lemma~\ref{lem:reduce_to_residue_field} is hopeless in practice since the computability of the quantity $N$
  comes from an argument involving an enumeration of proofs in some proof calculus.
  The same holds true for the algorithm that we obtain 
  in Theorem \ref{thm:main}.
\end{remark}

\begin{remark}
  The version of Lemma \ref{lem:reduce_to_residue_field} given in \cite[Proposition 3.5]{DS} is significantly stronger, being phrased for an arbitrary equicharacteristic excellent henselian local domain instead of an equicharacteristic henselian discrete valuation ring, and allowing finitely many parameters.
  In the proof, this comes down to replacing Greenberg's theorem by a more general version of Artin approximation.

  On the other hand, note that we can dispense with the excellence condition on the valuation ring $R$.
  This is essentially due to us only allowing parameters from $\FF_p[\pi]$.
\end{remark}

\subsection*{Step 2: Existential $\Lring(\varpi)$-sentences in $R$, with a regularity condition}

For this step, let us fix $f_1, \dotsc, f_n \in \FF_p[\pi][X_1, \dotsc, X_m]$ such that the system of equations $\bigwedge_{i=1}^n f_i(\underline X) = 0$ describes an affine $\FF_p(\pi)$-variety $V$, i.e.\ the ideal 
$$
 \mathfrak{f}:=(f_1, \dotsc, f_n) \unlhd \FF_p(\pi)[X_1, \dotsc, X_m]
$$ 
is prime.
Let $d$ be the dimension of $V$.

Recall that a solution $\underline x$ to the system of equations in some extension field $F/\FF_p(\pi)$ is called \emph{smooth} if the Jacobian condition is satisfied, i.e.\ the matrix $(\frac{\partial f_i}{\partial X_j}(\underline x))_{ij}$ has rank $m-d$.

\begin{lemma}\label{lem:henselian_dense}
  Let $g \in \FF_p[\pi][X_1, \dotsc, X_m]$ be a polynomial which is not in the ideal $\mathfrak{f}$, so that $g$ does not vanish identically on the variety $V$.
If there is a solution $\underline x \in R^m$ to the system of equations $\bigwedge_{i=1}^n f_i(\underline X) = 0$ such that the $K$-point described by $\underline x$ is smooth, then there exists such a solution with $g(\underline x) \neq 0$.
\end{lemma}
\begin{proof}
  This is \cite[Theorem 2.4]{DS}, or at least the special case of it which we will need below.
  We only briefly describe the idea of the proof:

  The statement comes down to an elaborate application of Hensel's lemma, or the implicit function theorem in henselian fields.
  Starting with the given solution $\underline x \in R^m$, perturb $d$ variables by a small amount (in the valuation topology) and solve for the remaining variables (which is where smoothness is required).
  By wisely choosing the initial perturbation, one can ensure that the inequation $g(\underline X) \neq 0$ becomes satisfied, since $g$ does not vanish identically on $V$.
  (For instance, one may reduce to the case where the $f_i$ describe an affine curve over $K$ by a generic hyperplane section argument, so the condition $g(\underline X) \neq 0$ is satisfied by all but finitely many points on the curve, in which case almost all small perturbations will be as required.)
\end{proof}

Below we want to work in a context where the affine $\FF_p(\pi)$-variety $V$ is not necessarily smooth (i.e.\ not all solutions to the equation system in all extension fields satisfy the Jacobian condition),
but instead satisfies the condition of regularity, i.e.\ the coordinate ring $\FF_p(\pi)[X_1, \dotsc, X_m]/\mathfrak{f}$ is a regular ring in the sense of commutative algebra.
In general, smooth varieties are regular, and over a perfect field the converse holds \cite[Corollary~4.3.32 and Corollary 4.3.33]{Liu}.
In our situation we still have the following:

\begin{lemma}\label{lem:regsmooth}
  Suppose that $V$ is regular.
  Then any solution $\underline x \in K^m$ to the system of equations $\bigwedge_{i=1}^{n}f_{i}(\underline X) = 0$ is smooth, i.e.\ satisfies the Jacobian condition.
\end{lemma}
\begin{proof}
  The field extension $K/\FF_p(\pi)$ is separable as $\pi$ does not have a $p$-th root in $K$ and hence the fields $K$ and $\FF_p(\pi)^{1/p}=\FF_p(\pi^{1/p})$ are linearly disjoint over $\FF_p(\pi)$.
  The tuple $\underline x$ describes an element of $V(K)$, i.e.\ a point $P$ of the scheme $V$ together with an $\FF_p(\pi)$-embedding of the residue field of $P$ into $K$.
  Thus the residue field of $P$ is separable over $\FF_p(\pi)$.
  As $P$ is regular by assumption, \cite[Proposition (17.15.1)]{EGA4_4} forces the point $P$ to be smooth, so the Jacobian condition holds.
\end{proof}

Now suppose that $V$ is regular, and $g \in \FF_p[\pi][X_1, \dotsc, X_m]$ is not in $\mathfrak{f}$.
In this situation, satisfaction of the statement $\exists \underline X (\bigwedge_{i=1}^n f_i(\underline X) = 0 \wedge g(\underline X) \neq 0)$ in $R$ can now be effectively reduced to an existential statement about the residue field $\kappa$:
By Lemma \ref{lem:henselian_dense} and Lemma \ref{lem:regsmooth}, the condition $g(\underline X) \neq 0$ can be dropped, and the remaining formula is handled by the previous step (Lemma \ref{lem:reduce_to_residue_field}).

\subsection*{Step 3: Arbitrary existential $\Lring(\varpi)$-sentences in $R$, using resolution of singularities}

\hspace{1cm}

Let $f_1, \dotsc, f_n, g \in \FF_p[\pi][X_1, \dotsc, X_m]$.
We wish to determine whether there exists $\underline x \in R^m$ with $f_i(\underline x) = 0$ for all $i = 1, \dotsc, n$, and $g(\underline x) \neq 0$.

We may assume that the ideal $\mathfrak{f}=(f_1, \dotsc, f_n) \unlhd \FF_p(\pi)[X_1, \dotsc, X_m]$ is prime, i.e.\ that the $f_i$ define an affine $\FF_p(\pi)$-variety $V$:
Indeed, by primary decomposition there exist ideals $I_1, \dotsc, I_k \unlhd \FF_p(\pi)[X_1, \dotsc, X_m]$ which are primary, so in particular the radicals $\sqrt{I_1}, \dotsc, \sqrt{I_k}$ are prime ideals, and $\mathfrak{f} = I_1 \cap \dotsb \cap I_k$.
Then a tuple $\underline x \in R^m$ lies in the vanishing locus of $f_1, \dotsc, f_n$ if and only if it lies in the vanishing locus of $\sqrt{I_j}$ for some $j = 1, \dotsc, k$, so we may replace the $f_i$ by a collection of generators for each of the $\sqrt{I_j}$, running all subsequent steps for all $j = 1, \dotsc, k$.
The computation of the ideals $\sqrt{I_j}$ (or rather, generating sets thereof) is completely explicit \cite{Seidenberg_ConstrInAlgebra}.

We can check algorithmically whether the $\FF_p(\pi)$-variety $V$ is regular:
For this we investigate the $\FF_p$-scheme described by the $f_i$ (i.e.\ we ``spread out'' $V$ to a scheme over $\FF_p[\pi]$  and hence a scheme over $\FF_p$, by interpreting $\pi$ as another indeterminate) and check whether the non-regular locus intersects the generic fibre of the map to $\operatorname{Spec}(\FF_p[\pi])$.
The regular locus can be computed explicitly since it coincides with the smooth locus over the perfect field $\FF_p$, which is given by explicit polynomial inequations (once the dimension of the scheme is computed).

Suppose that $V$ is regular.
We can check algorithmically whether $g$ vanishes identically on the zero locus of the $f_i$, i.e.\ whether $g$ is contained in $\mathfrak{f}$.
If yes, then there certainly is no solution to $\bigwedge_{i=1}^n f_i(\underline X) = 0 \wedge g(\underline X) \neq 0$.
Otherwise, Step 2 is applicable.
In other words, we can reduce to $\kappa$ the question of whether there is a tuple in $R$ as required.

Suppose now that $V$ defined by the $f_i$ is not regular.
Assuming hypothesis \eqref{R:Blow}, there exists a morphism $V' \to V$ which is a blowing-up along some closed proper subscheme, and where $V'$ is a regular $\FF_p(\pi)$-variety.
The morphism $V' \to V$ is projective \cite[Proposition 8.1.22]{Liu}, and so we can see $V'$ as a closed subscheme of $\mathbb{P}_V^{m'}$ for suitable $m'$.
Thus $V'$ (usually not affine) is described by finitely many polynomials $h_1, \dotsc, h_{n'} \in \FF_p(\pi)[X_1, \dotsc, X_m, Y_0, \dotsc, Y_{m'}]$ which are homogeneous in the variables $Y_j$ (but not usually in the $X_i$), and include the polynomials $f_1, \dotsc, f_n$ defining $V$.
We may assume that the defining polynomials $h_i$ have coefficients in $\FF_p[\pi]$ by scaling suitably.
A point of $V'$ (in some field containing $\FF_p(\pi)$) is given by a zero $(\underline x, \underline y)$ of the $h_i$, where $\underline y$ is not the zero tuple, and we identify solution tuples that differ only by a scaling of $\underline y$.
The morphism $V' \to V$ is given on points by simply dropping the variables $Y_0, \dotsc, Y_{m'}$.

We can cover $V'$ by affine open subvarieties $V'_0, \dotsc, V'_{m'}$, where $V'_i$ is obtained from the description of $V'$ by setting $Y_i$ to $1$, i.e.\ dehomogenizing at $Y_i$.
This eliminates the need to work with non-affine varieties.

The morphism $V' \to V$ is birational as a blowing-up along a closed proper subscheme, thus it is an isomorphism above some dense open $U \subseteq V$.
Concretely, let us suppose that $V_0' \neq \emptyset$ by permuting the variables $Y_i$ if necessary, so that $V_0' \subseteq V'$ is a dense open subvariety.
Then $V' \to V$ being birational means that there is $u \in \FF_p[\pi][X_1, \dotsc, X_m]$, not in $\mathfrak{f}$, such that $V'_0 \to V$ is an isomorphism above the open set $U \subseteq V$ defined by $u(\underline X) \neq 0$, i.e.\ the ring homomorphism 
$$
 \varphi_u \colon \FF_p(\pi)[X_1, \dotsc, X_m, 1/u(\underline X)]/(f_1, \dotsc, f_n) \to \FF_p(\pi)[X_1, \dotsc, X_m, Y_0, \dotsc, Y_{m'}, 1/u(\underline X)]/(h_1, \dotsc, h_{n'}, Y_0 - 1)
$$
is an isomorphism.

Assuming the existence of $V' \to V$ as above, i.e.\ projective and birational with $V'$ regular, such a morphism can in fact be found algorithmically:
To do this, one simply searches exhaustively for all the defining data ($m'$, polynomials $h_1, \dotsc, h_{n'}$, $u$, and an inverse $\psi_u$ to the ring homomorphism $\varphi_u$ above), and checks that these data satisfy the conditions, i.e.\ the $V'_i$ are all regular and $\psi_u$ is indeed an inverse to $\varphi_u$.

Let $S$ be the set of
tuples $\underline x \in R^m$ satisfying $f_i(\underline x) = 0$ for all $i$ and $g(\underline x) \neq 0$.
To check whether $S=\emptyset$, we now proceed as follows.
First check whether there is some tuple $(\underline x, \underline y) \in R^m \times R^{m'}$ which gives a $K$-point of one of $V'_0, \dotsc, V'_{m'}$ and satisfies $g(\underline x) \neq 0$.
This can be accomplished using the work already done, since each $V'_i$ is a regular affine  $\FF_p(\pi)$-variety.
If such a tuple $(\underline x, \underline y)$ exists, then $\underline x\in S$, and we are done.

Hence let us suppose that no such tuple $(\underline x, \underline y)$ exists.
Then there exists no $\underline x \in R^m$ with $g(\underline x) \neq 0$ which gives a $K$-point of $U$:
Namely, if there were such $\underline x$ which was additionally a $K$-point of $U$, then we could lift to a $K$-point $(\underline x, \underline y) \in K^{m+m'+1}$ of $V'$, and by scaling we could take $\underline y$ to be in $R^{m'+1}$ with one coordinate equal to $1$, i.e.\ $(\underline x, \underline y)$ would describe a point on some $V'_i$.

Thus any $\underline x \in S$ must lie in the lower-dimensional $V \setminus U$.
We continue the algorithm by replacing $V$ with the algebraic set $V \setminus U$, i.e.\ by adding $u$ to the list of polynomials $f_1, \dotsc, f_n$.
Since $V \setminus U$ has lower dimension than $V$, this procedure eventually terminates.

We have shown how to determine whether the system $\bigwedge_{i=1}^n f_i(\underline X) = 0 \wedge g(\underline X) \neq 0$ has a solution in $R$.
The question of whether any given existential $\Lring(\pi)$-sentence holds in $R$ can be effectively reduced to a disjunction of finitely many systems of this form, using the disjunctive normal form and replacing conjunctions of inequations $g_1(\underline X) \neq 0 \wedge g_2(\underline X) \neq 0$ by a single inequation $(g_1 \cdot g_2)(\underline X) \neq 0$.
We have proven:
\begin{proposition}[{\cite[Theorem 4.3]{DS}}]\label{prp:DS}
  Assume \eqref{R:Blow}.
  The existential $\Lring(\varpi)$-theory of $(R,\pi)$ is decidable relative to the existential $\Lring$-theory of the residue field $\kappa$.
\end{proposition}
Theorem \ref{thm:DS2} follows since the $\Lval(\varpi)$-structure $(K, v,\pi)$, 
where $v$ is the valuation with valuation ring $R$, 
is quantifier-freely interpretable in the $\mathcal{L}_{\rm ring}(\varpi)$-structure $(R,\pi)$.

\begin{remark}
  Comparing with the phrasing in \cite[Theorem 4.3]{DS}, it is clear in our presentation that the assumption of excellence of $R$ may be dropped.
  The phrasing in \cite[Theorem 4.3]{DS} also only states that the $\Lring(\varpi)$-theory of $(R,\pi)$ is decidable ``relative to the existential theory of $\kappa$ and the [$\Lring(\varpi)$]-diagram of $R$''.
  However, it is unclear what is meant by this diagram.
  In particular, if $R$ is uncountable, then the diagram of $R$ in its usual model-theoretic meaning is an uncountable object, and hence cannot be made available to any algorithm as an oracle.
  In any case, the presentation above shows that nothing is required besides the existential $\Lring$-theory of $\kappa$.
\end{remark}

\begin{remark}
  Above we have taken care to perform all computation in polynomial rings over the concrete field $\FF_p(\pi)$ (or over the ring $\FF_p[\pi]$).
  Denef and Schoutens are not always clear on this point.
  For instance, \cite[Lemma 4.2]{DS} makes a claim about an algorithm computing a reduction over the ring $R$, but it is unclear what this means if $R$ is not explicitly presented (for example uncountable).

  While it may appear that such problems can be easily avoided by working over finitely generated subfields (and indeed this is suggested in \cite[Remark 4.1]{DS}), this leads to rather subtle issues in general, since neither reducedness nor regularity is preserved under inseparable base change:
  If $K_0 \subseteq K$ is a subfield over which $K$ is not separable, then a regular variety over $K_0$ may become non-reduced (thus in particular non-regular) over $K$.
  Similarly, one needs to be careful in general to distinguish regularity from smoothness, because these disagree for varieties over imperfect base fields, and so regularity cannot be checked using the Jacobian criterion as suggested in \cite[Remark 4.1]{DS}.

  We can avoid these issues above because the extension $K/\FF_p(\pi)$ is separable and therefore none of the pathologies mentioned can arise.
\end{remark}

\begin{remark}
  Above we have in fact only used that there exists a resolution of the affine variety $V$ by a projective birational morphism, rather than a blowing-up.
  This looks superficially weaker than hypothesis \eqref{R:Blow} or the equivalent \cite[Section 4, Conjecture 1]{DS}, but in fact any projective birational morphism to a quasi-projective variety is a blowing-up \cite[Theorem 8.1.24]{Liu}.
  (Of course, it suffices for us to demand \eqref{R:Blow} for affine varieties.)

  In fact, the argument above can be modified to make do with \eqref{R:Mod} in place of \eqref{R:Blow}.
  To do this, consider the $\FF_p[\pi]$-scheme $\mathcal{V} = \operatorname{Spec}(\FF_p[\pi][X_1, \dotsc, X_m]/(f_1, \dotsc, f_n))$, which has $V$ as above as its generic fibre.
  Then a birational proper morphism $V' \to V$ with $V'$ regular (which exists under the assumption of \eqref{R:Mod}) can be extended to a proper morphism $\mathcal{V}' \to \mathcal{V}$, with an $\FF_p[\pi]$-scheme $\mathcal{V}'$ which has $V'$ as its generic fibre.
  (This step seems to require Nagata's compactification theorem, cf.~\cite[Section (12.15)]{GoertzWedhorn}.)
  Covering $\mathcal{V}'$ by affine $\FF_p[\pi]$-schemes $\mathcal{V}'_0, \dotsc, \mathcal{V}'_{m'}$, one can then proceed as above.
  We omit the details.
\end{remark}

\begin{remark}
  In spite of the technical superiority of scheme-theoretic algebraic geometry over more naive conceptions of varieties, we find that one really can work in the naive language almost throughout, excepting some of the more subtle points surrounding smoothness and regularity which cannot be handled adequately in this way.

  Denef and Schoutens largely use scheme-theoretic language, but not entirely correctly.
  In particular, by an open $W$ of a finite type $R$-scheme $X$ as in \cite[Theorem 2.4 and proof of Theorem 4.3]{DS} they seem to mean a subset of the $R$-points of $X$ defined by finitely many inequations, which is not the same as an open subscheme (which, after all, would itself be a finite type $R$-scheme and could simply replace $X$), but rather an open subscheme of the generic fibre $X_K$.
  An $R$-rational point of $W$ in their language is then an $R$-point of $X$ such that the associated $K$-point of $X_K$ lies on $W$.
  This issue also seems to have been noted independently in \cite[proof of Theorem A]{Kostas_tame}.
\end{remark}

\section{Axiomatization of the existential theory}
\label{sec:axiomatization}

\noindent
In this final section we prove
Theorem~\ref{thm:main},
which is a strengthening of
Theorem~\ref{thm:intro_main}
to allow
parameters from a $\mathbb{Z}$-valued base field $(C,u)$.
Here, by $\mathbb{Z}$-valued we mean that $uC\cong\mathbb{Z}$
(sometimes called {\em discrete}).
A major role in our proof is played by partial sections of residue maps:

\begin{definition}\label{def:partial_section}
Let $(K,v)$ be a valued field.
We denote by ${\rm res}_v \colon \mathcal{O}_v\rightarrow Kv$ the residue map.
A {\em partial section} of ${\rm res}_v$ is a ring homomorphism
$\zeta \colon k\rightarrow\mathcal{O}_v$ with ${\rm res}_v\circ\zeta={\rm id}_k$
for some subfield $k$ of $Kv$.
A {\em section} of ${\rm res}_v$ is a partial section
$\zeta \colon Kv\rightarrow\mathcal{O}_v$ defined on the full residue field $Kv$.
\end{definition}

We will repeatedly have to extend such partial sections.
There are two main settings where this is possible.
The first one is in the context of complete discrete valuation rings:

\begin{proposition}\label{prop:complete_structure}
Let $(K,v)$ be an equicharacteristic complete $\mathbb{Z}$-valued field with uniformizer $\pi$.
Every partial section $\zeta\colon k\rightarrow K$ of ${\rm res}_v$
has a unique extension to an embedding
of valued fields
$\alpha\colon (k(\!(t)\!),v_t)\rightarrow(K,v)$
with $\alpha(t)=\pi$,
which is an isomorphism if $k=Kv$.
\end{proposition}

\begin{proof}
The partial section $\zeta$ extends uniquely to an embedding $\alpha_0\colon (k(t),v_t)\rightarrow(K,v)$ with $\alpha_0(t)=\pi$,
and then the existence and uniqueness of the completion give a unique extension $\alpha$ of $\alpha_0$ from the completion $(k(\!(t)\!),v_t)$ of $(k(t),v_t)$
into the complete valued field $(K,v)$.
If $\zeta$ is a section, every element of $K$ has a power series expansion in $t$ with coefficients in $\zeta(K)$, see e.g.~\cite[Chapter II \S4]{Serre},
hence $\alpha$ is surjective.
\end{proof}

\begin{proposition}\label{prop:complete_lift}
Let $(K, v)$ be an
equicharacteristic complete $\mathbb{Z}$-valued field.
Every partial section $\zeta \colon k \to \mathcal{O}_v$ of $\mathrm{res}_v$ with $Kv/k$ separable extends to a section of $\mathrm{res}_v$.
In particular, $\mathrm{res}_v$ has a section.
\end{proposition}
\begin{proof}
  This follows from the structure theory of complete discrete valuation rings,
  e.g.~apply \cite[Satz 1a]{Roquette} with $E_0 = \zeta(k)$ in the notation there.
  The ``in particular'' follows by applying the first part to the partial section defined on the prime field.
\end{proof}

Secondly, we can always extend partial sections after passage to a suitable elementary extension:

\begin{lemma}\label{lem:section}
Let $(K,v)$ be a henselian valued field.
For every partial section $\zeta_{0} \colon k_{0}\rightarrow K$ of ${\rm res}_v$
with $Kv/k_0$ separable
there exists an elementary extension
$(K,v)\prec(K^{*},v^{*})$
with a partial section
$\zeta \colon Kv\rightarrow K^{*}$ of ${\rm res}_{v^*}$
that extends $\zeta_{0}$.
\end{lemma}

\begin{proof}
For each subfield $E\subseteq Kv$,
let $\mathcal{L}_{\mathrm{val}}(E,Kv)$ be the language expanding $\mathcal{L}_{\mathrm{val}}$ by new constant symbols
$\{c_{x}\mid x\in E\}$
and
$\{d_{x}\mid x\in Kv\}$.
Let $\rho(a,b)$ be an $\mathcal{L}_{\rm val}$-formula expressing that $a$ and $b$ have the same residue
and
consider the $\mathcal{L}_{\mathrm{val}}(E,Kv)$-theory
\begin{align*}
	T_{E}
&:=
	\{d_{x}\in\mathcal{O}\mid x\in Kv\}\cup
	\{c_{x}\in\mathcal{O}\wedge\rho(c_x,d_x)\mid x\in E\}\cup\{c_{x}+c_{y}=c_{x+y}\wedge c_{x}c_{y}=c_{xy}\mid x,y\in E\}.
\end{align*}
If we expand $(K,v)$ to an $\mathcal{L}_{\mathrm{val}}(k_0,Kv)$-structure $\mathbb{K}_{\zeta_0}$ by interpreting
$c_{x}^{\mathbb{K}_{\zeta_0}}=\zeta_{0}(x)$ 
for $x\in k_{0}$ and
$d_{x}^{\mathbb{K}_{\zeta_0}}$ any choice of element in the valuation ring of $v$ with residue $x$, for each $x\in Kv$,
then $\mathbb{K}_{\zeta_0}\models T_{k_0}$.

Let $\rmDelta$ be a finite subset of $T_{Kv}$.
Then there exists a finitely generated subextension $E/k_{0}$ of $Kv/k_{0}$ such that $\rmDelta$ is an $\mathcal{L}_{\mathrm{val}}(E,Kv)$-theory.
Since $Kv/k_{0}$ is separable, $E/k_{0}$ is separably generated.
Therefore, by \cite[Lemma 2.3]{AF},
$\zeta_{0}$ extends to a partial section
$\hat{\zeta} \colon E\rightarrow K$
of ${\rm res}_v$.
Equivalently,
there is an $\mathcal{L}_{\mathrm{val}}(E,Kv)$-expansion $\mathbb{K}_{\hat{\zeta}}$ of $\mathbb{K}_{\zeta_0}$ to
a model of $T_{E}$.
In particular, $\mathbb{K}_{\hat{\zeta}}$ models $\rmDelta$.

Therefore, $T_{Kv}$ is consistent with the elementary diagram of $\mathbb{K}_{\zeta_0}$,
hence 
by the compactness theorem
there exists an
$\mathcal{L}_{\mathrm{val}}({Kv,Kv})$-structure $\mathbb{K}^*\models T_{Kv}$
whose $\mathcal{L}_{\mathrm{val}}({k_0,Kv})$-reduct is an elementary extension of $\mathbb{K}_{\zeta_0}$.
In particular,
its $\mathcal{L}_{\mathrm{val}}$-reduct $(K^*,v^*)$ is an elementary extension of $(K,v)$,
and
$\zeta \colon Kv\rightarrow K^*$ defined by $x\mapsto c_x^{\mathbb{K}^*}$
is a partial section of ${\rm res}_{v^*}$ that extends $\zeta_0$.
\end{proof}

\begin{proposition}\label{lem:section_full}
Let $(K,v)$ be a henselian valued field.
For every partial section $\zeta_{0} \colon k_{0}\rightarrow K$ of ${\rm res}_v$
with $Kv/k_0$ separable
there exists an elementary extension
$(K,v)\prec(K^{*},v^{*})$
with a section
$\zeta \colon K^*v^*\rightarrow K^{*}$ of ${\rm res}_{v^*}$
that extends $\zeta_{0}$.
\end{proposition}

\begin{proof}
By Lemma \ref{lem:section} 
there exists an elementary extension $\mathbb{K}_0:=(K,v)\prec\mathbb{K}_1=(K_1,v_1)$ with a partial section
$\zeta_1 \colon Kv\rightarrow K_1$ of ${\rm res}_{v_1}$ extending $\zeta_0$.
Since $Kv\prec K_1v_1$, the extension $K_1v_1/Kv$ is in particular separable, cf.~\cite[Corollary 3.1.3]{Ershov_book},
and so we can iterate Lemma \ref{lem:section} to obtain a chain
of elementary extensions $\mathbb{K}=\mathbb{K}_0\prec\mathbb{K}_1\prec\dots$
and for each $i>0$ a partial section $\zeta_i \colon K_{i-1}v_{i-1}\rightarrow K_i$ of ${\rm res}_{v_i}$
extending $\zeta_{i-1}$.
The direct limit $\mathbb{K}^*=(K^*,v^*):=\varinjlim_i\mathbb{K}_i$ is then an elementary extension of $\mathbb{K}_{0}$
with a partial section $K^*v^*=\varinjlim_i K_{i-1}v_{i-1}\rightarrow \varinjlim_i K_i=K^*$ extending $\zeta_0$.
\end{proof}

Given a field extension $K/C$ we denote by $(K,C)$ the $\mathcal{L}_{\mathrm{ring}}(C)$-structure expanding $K$ in the usual way: the constant symbol $c_{x}$ is interpreted by $x$ itself, for each $x\in C$.
Analogously, $(K,v,C)$ denotes the $\mathcal{L}_{\mathrm{val}}(C)$-structure similarly expanding $(K,v)$,
and $(K,v,\pi,C)$ denotes the $\mathcal{L}_{\mathrm{val}}(\varpi,C)$-structure similarly expanding $(K,v,\pi)$.

\begin{corollary}\label{cor:large}
Assume \eqref{R:ExClosed}.
Let $(K,v)$ be an equicharacteristic henselian valued field with $Kv$ large,
and assume that $v$ is trivial on a subfield $C$ of $K$.
We identify $C$ with its image under ${\rm res}_v$ and assume that
the extension
$Kv/C$ is separable.
Then the existential $\mathcal{L}_{\rm ring}(C)$-theories of $K$ and $Kv$
coincide.
\end{corollary}

\begin{proof}
By Proposition \ref{lem:section_full}
there is an elementary extension $(K,v)\prec(K^*,v^*)$ with a section $\zeta \colon K^*v^*\rightarrow K^*$ of ${\rm res}_{v^*}$ that extends the partial section ${\rm id}_C$ of ${\rm res}_v$.
Since $Kv$ is large, so is its elementary extension $K^*v^*$,
hence $\zeta(K^*v^*)\prec_\exists K^*$ by \eqref{R:ExClosed}.
So for the existential $\mathcal{L}_{\rm ring}(C)$-theories we get
${\rm Th}_\exists(K)={\rm Th}_\exists(K^*)={\rm Th}_\exists(\zeta(K^*v^*))={\rm Th}_\exists(K^*v^*)={\rm Th}_\exists(Kv)$.
\end{proof}

After these preparations on extending partial sections we are now ready to attack the proof of Theorem \ref{thm:intro_main}.

\begin{definition}\label{def:T}
We denote by $T$ the $\mathcal{L}_{\rm val}(\varpi)$-theory of equicharacteristic henselian valued fields
with distinguished uniformizer.
\end{definition}

\begin{remark}
The class of equicharacteristic henselian valued fields with distinguished uniformizer is elementary.
Moreover, $T$ admits a recursive axiomatization,
for example see
\cite[Section 4]{KuhlmannTame}
for an explicit axiomatization of henselianity.
\end{remark}

\begin{lemma}\label{lem:reduce_to_Lring}
Every existential $\mathcal{L}_{\rm val}(\varpi)$-formula
is equivalent modulo $T$ to an existential
$\mathcal{L}_{\rm ring}(\varpi)$-formula.
\end{lemma}

\begin{proof}
This follows since the existential $\mathcal{L}_{\rm ring}(\varpi)$-formula $\varphi(x)$
given by $\exists y\; y^2+y=\varpi  x^2$
defines the valuation ring $\mathcal{O}_v$ in every $(K,v,\pi)\models T$, and then so does the universal formula
$\neg\varphi((\varpi x)^{-1})$
given by $\forall y\;\varpi x^2(y^2+y)\neq1$.
These formulas basically go back to Julia Robinson \cite{JR},
see \cite{FehmJahnke} for more on definable valuations.
\end{proof}

Let $R$ be a discrete valuation ring with field of fractions $C$, $u$ the valuation on $C$ with $\mathcal{O}_u = R$, and $(\hat C, \hat u)$ the completion of $(C, u)$.
Then $R$ is \emph{excellent} if and only if $\hat{C}/C$ is separable.
This is compatible with the more general definition of excellent rings in commutative algebra, see \cite[Corollary 8.2.40(b)]{Liu}, but for discrete valuation rings we may as well take it as our definition.
\begin{remark}
  Although we will not need it below, we note that $R$ is excellent if and only if it is defectless in the sense of valuation theory, i.e.\ the fundamental inequality is an equality for all finite extensions of $C$, see \cite[(1.5)]{KuhlmannTame}.
  Indeed, $\hat C/C$ is separable if and only if $R$ is a Japanese ring by \cite[Corollaire (7.6.6)]{EGA4_2}, which in turn is the case if and only if $R$ is defectless by \cite[Chapitre VI §8 No 5 Théorème 2]{Bourbaki_CommAlg567}.
\end{remark}

\begin{proposition}\label{prop:main}
Assume \eqref{R:ExClosed}.
Let $(C,u)$ be an equicharacteristic $\mathbb{Z}$-valued field with distinguished uniformizer $\pi$ 
such that $\mathcal{O}_u$ is excellent.
Let $(K,v,\pi),(L,w,\pi)\models T$ be extensions of $(C,u,\pi)$ such that
$Kv/Cu$ and $Lw/Cu$ are separable.
If ${\rm Th}_\exists(Kv,Cu)\subseteq{\rm Th}_\exists(Lw,Cu)$, then
${\rm Th}_\exists(K,v,\pi,C)\subseteq{\rm Th}_\exists(L,w,\pi,C)$.
\end{proposition}

\begin{proof}
Let $\varphi$ be an existential  $\mathcal{L}_{\rm val}(\varpi,C)$-sentence with $(K,v,\pi,C)\models\varphi$. We have to show that $(L,w,\pi,C)\models\varphi$,
and for this we are allowed to replace both $(K,v,\pi,C)$ and $(L,w,\pi,C)$ by arbitrary elementary extensions,
so assume without loss of generality that $(K,v,\pi,C)$ is
$\aleph_{1}$-saturated and
$(L,w,\pi,C)$ is $|K|^+$-saturated.
Since in particular $Lw$ 
is at least $|Kv|^+$-saturated,
and $(Lw,Cu)\models\mathrm{Th}_{\exists}(Kv,Cu)$,
there is a $Cu$-embedding $f \colon Kv\rightarrow Lw$,
cf.~\cite[Lemma 5.2.1]{ChangKeisler}.
By Lemma~\ref{lem:reduce_to_Lring}
we can assume without loss of generality that $\varphi$ is an existential 
$\mathcal{L}_{\rm ring}(\varpi,C)$-sentence.
Furthermore, we can replace each occurrence of $\varpi$ by the constant symbol $c_\pi$
to assume without loss of generality that $\varphi$ is an existential
$\mathcal{L}_{\rm ring}(C)$-sentence.

Since $\pi$ is a uniformizer of $v$, $v(\pi)$ generates a convex subgroup of $vK$ isomorphic to $\mathbb{Z}$.
In particular, 
$v$ has a finest proper coarsening $v^+$.
As $v^+$ is trivial on $C$, we can view 
$Kv^+$ as an extension of $C$.
On $Kv^+$,
$v$ induces a henselian valuation $\bar{v}$ extending $u$ with $\bar{v}(Kv^{+})\cong\mathbb{Z}$ and with uniformizer ${\pi}$.
Since $(K,v)$ 
is $\aleph_1$-saturated,
$(Kv^+,\bar{v})$ is complete, see \cite[Claim on page 411]{AK} and note that
maximal $\mathbb{Z}$-valued fields are complete.
Given that $\mathcal{O}_u$ is excellent, 
that $(Kv^+, \bar{v})$ and $(C, u)$ have a common uniformizer, 
and that the residue extension $Kv/Cu$ of $(Kv^+,\bar{v})/(C, u)$ is separable, 
we have that $Kv^+/C$ is separable by \cite[proof of Lemma 3.6/2]{BLR}\footnote{The ``in particular'' of the statement of the lemma is not correct, but this does not concern us.
Alternatively, the separability of $Kv^+/C$ follows from applying \cite[p.~44, Lemma]{Ershov} to the henselization of $(C, u)$.}.
Since moreover $Kv^+$ is large, because it admits the nontrivial henselian valuation $\bar{v}$, Corollary~\ref{cor:large} (which works under \eqref{R:ExClosed}) shows
that ${\rm Th}_\exists(K,C)={\rm Th}_\exists(Kv^+,C)$, and the analogous argument (note that $|K|^+\geq\aleph_1$) gives that
${\rm Th}_\exists(L,C)={\rm Th}_\exists(Lw^+,C)$.

Since we assumed that $\varphi\in{\rm Th}_\exists(K,C)$ and want to show that
$\varphi\in{\rm Th}_\exists(L,C)$,
we can now replace $(K,v,\pi,C)$ by $(Kv^+,\bar{v},{\pi},C)$ and $(L,w,\pi,C)$ by $(Lw^+,\bar{w},{\pi},C)$,
so that both
$(K,v)$ and $(L,w)$ are complete $\mathbb{Z}$-valued fields
with uniformizer $\pi$
---
note however that we may no longer assume any saturation for either $(K,v)$ or $(L,w)$.
We may then also
replace $C$ by $\hat{C}$
to assume that $(C,u)$ is complete.
By Proposition \ref{prop:complete_lift} there exists a section $\zeta \colon Cu\rightarrow C$ of ${\rm res}_u$, which we can see as a partial section of $\mathrm{res}_v$ and $\mathrm{res}_w$.
Since $Kv/Cu$ and $Lw/Cu$ are separable, applying Proposition \ref{prop:complete_lift}
again yields that
$\zeta$ extends to a section $\xi_K \colon Kv\rightarrow K$ of ${\rm res}_v$,
and to a section
$\xi_L \colon Lw\rightarrow L$ of ${\rm res}_w$.
By Proposition \ref{prop:complete_structure},
$\zeta$, $\xi_K$ and $\xi_L$ extend to isomorphisms
$\gamma \colon (Cu(\!(t)\!),v_t)\rightarrow(C,u)$,
$\alpha \colon (Kv(\!(t)\!),v_t)\rightarrow (K,v)$ and
$\beta \colon (Lw(\!(t)\!),v_t)\rightarrow (L,w)$ with
$\gamma(t)=\pi$,
$\alpha(t)=\pi$ and
$\beta(t)=\pi$,
and $\alpha|_{Cu(\!(t)\!)}=\beta|_{Cu(\!(t)\!)}=\gamma$
by the uniqueness assertion thereof.
Moreover, the $Cu$-embedding $f \colon Kv\rightarrow Lw$ extends uniquely to an embedding
$\iota \colon (Kv(\!(t)\!),v_t)\rightarrow (Lw(\!(t)\!),v_t)$
with $\iota(t)=t$
that is the identity on $Cu(\!(t)\!)$.
Therefore, $\beta\circ\iota\circ\alpha^{-1} \colon (K,v)\rightarrow(L,w)$ is an embedding which restricted to $C$ is $\gamma\circ{\rm id}_{Cu(\!(t)\!)}\circ\gamma^{-1}={\rm id}_{C}$.
\begin{align*}
\renewcommand{\labelstyle}{\textstyle}
    \xymatrix@=2em{
        K
        &&
        Kv(\!(t)\!)
        \ar@{->}[ll]_{\alpha}
        \ar@{^{(}->}[rr]^{\iota}
        &&
        Lw(\!(t)\!)
        \ar@{->}[rr]^{\beta}
        &&
        L
        \\
        &&
        &&
        &&
        \\
        C
        \ar@{-}[uu]_{}^{}
        &&
        Cu(\!(t)\!)
        \ar@{-}[uu]_{}^{}
        \ar@{->}[ll]_{\gamma}
        \ar@{->}[rr]^{\mathrm{id}}
        &&
        Cu(\!(t)\!)
        \ar@{-}[uu]_{}^{}
        \ar@{->}[rr]^{\gamma}
        &&
        C
        \ar@{-}[uu]_{}^{}
    }
\end{align*}

So $(K,C)\models\varphi$ implies that 
$(L,C)\models\varphi$, which concludes the proof.
\end{proof}

\begin{theorem}\label{thm:main}
Assume \eqref{R:ExClosed}.
Let $(C,u)$ be an equicharacteristic $\mathbb{Z}$-valued field
with uniformizer $\pi$
such that $\mathcal{O}_u$ is excellent,
and let $(K,v)$ be a henselian
extension of $(C,u)$
for which $\pi$ is a uniformizer and $Kv/Cu$ is separable.
Then the universal/existential $\mathcal{L}_{\rm val}(\varpi,C)$-theory of $(K,v,\pi,C)$ is entailed by
\begin{enumerate}[(i)]

    \item the quantifier-free diagram of 
    the $\mathcal{L}_{\rm val}(\varpi)$-structure $(C,u,\pi)$,
     \item the $\mathcal{L}_{\rm val}$-axioms for equicharacteristic henselian valued fields,
    
    \item the $\mathcal{L}_{\rm val}(\varpi)$-axiom expressing that $\pi$ has smallest positive value, and
   \item $\mathcal{L}_{\rm val}(C)$-axioms expressing that the residue field models the universal/existential $\mathcal{L}_{\rm ring}(Cu)$-theory of $Kv$.
\end{enumerate} 
In particular, 
if $C$ is countable and we fix a surjection $\rho \colon \mathbb{N}\rightarrow C$,
the existential $\mathcal{L}_{\rm val}(\varpi,C)$-theory of $(K,v,\pi,C)$
is decidable relative to the existential $\mathcal{L}_{\rm ring}(Cu)$-theory of $(Kv,Cu)$
and the quantifier-free diagram of the $\mathcal{L}_{\rm val}(\varpi)$-structure $(C,u,\pi)$.
\end{theorem}

\begin{proof}
Note that $(K,v,\pi)$ is a model of 
the theory $T$ from Definition \ref{def:T}.
Every model of $(i)$-$(iii)$
can be viewed as a henselian valued field
extending $(C,u)$ with uniformizer $\pi$,
in particular as a model of $T$.
Let $(L,w,\pi,C)$ be a model of $(i)$-$(iv)$.
Since $Kv/Cu$ is separable,
the universal $\mathcal{L}_{\rm ring}(Cu)$-theory of $Kv$ contains for every finite $p$-independent tuple $\underline x$ from $Cu$ the statement that $\underline x$ remains $p$-independent over $Kv$, so since $(L,w)$ satisfies $(iv)$, also $Lw/Cu$ is separable.
So since $(iv)$ for $(L,w)$ also
implies that 
${\rm Th}_\exists(Lw,Cu)={\rm Th}_\exists(Kv,Cu)$,
Proposition \ref{prop:main}
gives that 
${\rm Th}_\exists(L,w,\pi,C)={\rm Th}_\exists(K,v,\pi,C)$.
This shows that every universal or existential sentence true in $(K,v,\pi,C)$ is entailed by $(i)$-$(iv)$.

The list of axioms $(ii)$ is recursive, and $(iii)$ is a single axiom.
Therefore, 
with an oracle for 
$(i)$ and an oracle for $(iv)$,
the sentences deducible from $(i)$-$(iv)$ can be recursively enumerated.
Since the negation of an existential sentence is universal,
and every existential or universal sentence true in $(K,v,\pi,C)$
is entailed by $(i)$-$(iv)$, 
the completeness theorem then gives a decision procedure for the existential theory of $(K,v,\pi,C)$ modulo these two oracles.
\end{proof}

\begin{remark}\label{rem:relatively-decidable}
The notion of relative decidability used here is the obvious modification of
Remark~\ref{rem:relatively-decidable_0}:
we define an injection of the set of $\mathcal{L}_{\rm val}(\varpi,C)$-sentences into $\mathbb{N}$
by a standard Gödel coding using the map
$C\rightarrow\mathbb{N}$, $c\mapsto\min\rho^{-1}(c)$.
Similarly, we define an injection 
of the set of $\mathcal{L}_{\rm ring}(Cu)$-sentences into $\mathbb{N}$
using in addition the map $Cu\rightarrow\mathbb{N}$,
$d\mapsto\min\rho^{-1}({\rm res}_u^{-1}(d))$.
Relative decidability then means that there exists a Turing machine that takes as input the code
of an existential $\mathcal{L}_{\rm val}(\varpi,C)$-sentence $\varphi$
and outputs yes or no according to whether
$(K,v,\pi,C)\models\varphi$,
and uses an oracle 
for the existential $\mathcal{L}_{\rm ring}(Cu)$-theory of $(Kv,Cu)$
and another oracle for the quantifier-free $\mathcal{L}_{\rm val}(\varpi,C)$-theory
of $(C,u,\pi,C)$.
\end{remark}

\begin{proof}[Proof of Theorem \ref{thm:intro_main}]\label{proof_of_15}
Let $F$ be the prime subfield of $K$.
Since $v(\pi)\neq0$ and $(K,v)$ is equicharacteristic,
$\pi$ is transcendental over $F$.
Since $v(\pi)>0$ and $v$ restricted to $F$ is trivial,
the restriction of $v$ to $F(\pi)$ is the $\pi$-adic valuation $v_{\pi}$.
Let $(C,u)=(F(\pi),v_{\pi})$.
Then $(K,v)$ is an extension of $(C,u)$.
Since $Cu=F$ is perfect, $Kv/Cu$ is separable.
The valuation ring $\mathcal{O}_{u}$ is excellent since $F(\!(\pi)\!)/F(\pi)$ is separable
(cf.~the proof of Lemma \ref{lem:regsmooth})

Let $(L,w,\pi')$ be a model of \ref{thm:intro_main}(i)--(iii).
Then $(L,w,\pi')$ is a model of \ref{thm:main}(ii)--(iii).
In particular, $(L,w,\pi')$ is equicharacteristic, and by \ref{thm:intro_main}(iii), 
$Lw$ has the same characteristic as $Kv$.
Thus, without loss of generality, $L$ is an extension of $F$.
By the same argument as above for $(K,v)$,
$(L,w)$ is an extension of $(F(\pi'),v_{\pi'})$.
The latter is isomorphic to $(F(\pi),v_{\pi})$,
thus $(L,w,\pi')$ admits a unique expansion to an $\mathcal{L}_{\mathrm{val}}(\varpi,C)$-structure
$\mathbb{L}$
that models \ref{thm:main}(i) for $\mathbb{K}:=(K,v,\pi,C)$.
Since $Lw$ models the universal/existential theory of $Kv$,
and $Cu=F$ is a prime field,
$\mathbb{L}$ models \ref{thm:main}(iv) for $\mathbb{K}$.

By Theorem~\ref{thm:main},
$\mathbb{L}$ models the universal/existential theory of $\mathbb{K}$.
In particular, $(L,w,\pi')$ models the universal/existential theory of $(K,v,\pi)$,
which proves the first part of the theorem.
The final sentence of the theorem follows, just as the final sentence of Theorem~\ref{thm:main}.
\end{proof}

\begin{proof}[Proof of Corollary \ref{cor:intro_Fpt}]
This follow from Theorem \ref{thm:main} with $(C,u)=(\mathbb{F}_q(t),v_t)$
and $(K,v)=(\mathbb{F}_q(\!(t)\!),v_t)$, for which we verify the prerequisites.
The existential $\mathcal{L}_{\rm ring}(\mathbb{F}_q)$-theory of $\mathbb{F}_q$ is decidable.
The same holds for the quantifier-free diagram of $(\mathbb{F}_q(t),v_t,t)$, since elements of $\mathbb{F}_q(t)$ can be explicitly represented as quotients of polynomials in such a way as to allow concrete computation of sums, products and the valuation under $v_t$.
More formally, the ring $\mathbb{F}_q(t)$ has an \emph{explicit representation} in the sense of \cite[§2]{EffectiveProcedures} by \cite[Theorem 3.4]{EffectiveProcedures}, and this representation in particular fixes a surjection $\rho \colon \mathbb{N} \to \mathbb{F}_q(t)$ with respect to which the quantifier-free diagram of $(\mathbb{F}_q(t), t)$ is decidable in the sense of Remark \ref{rem:relatively-decidable}.
Since elements of $\mathbb{F}_q(t)$ are represented by quotients $f/g$ with $f, g \in \mathbb{F}_q[t]$ with $g \neq 0$ in the proof of \cite[Theorem 3.4]{EffectiveProcedures}, and one can determine $v_t(f)$ and $v_t(g)$ for explicitly given $f$ and $g$, it is decidable whether a given element of $\mathbb{F}_q(t)$ lies in the valuation ring of $v_t$.
\end{proof}

\begin{corollary}\label{cor:ex_AKE}
Assume \eqref{R:ExClosed}.
Let $(C,u)\subseteq(K,v)$ be an extension of equicharacteristic henselian valued fields.
Suppose that $uC=\mathbb{Z}$ and that $\mathcal{O}_u$ is excellent.
If $uC\prec_\exists vK$ and $Cu\prec_\exists Kv$, then
$(C,u)\prec_\exists(K,v)$.
\end{corollary}

\begin{proof}
The assumption $\mathbb{Z}=uC\prec_\exists vK$ implies
that a uniformizer of $u$ is also a uniformizer of $v$.
The assumption $Cu\prec_\exists Kv$ implies that $Kv/Cu$ is separable
(see again \cite[Corollary 3.1.3]{Ershov_book}) and that $Cu$ and $Kv$ have the same existential $\mathcal{L}_{\rm ring}(Cu)$-theory.
Theorem \ref{thm:main} then gives that
$(C,u)$ and $(K,v)$ have the same existential $\mathcal{L}_{\rm val}(C)$-theory.
\end{proof}

\begin{remark}
The assumption that $\mathcal{O}_u$ is excellent is necessary in all these statements.
For example, in Corollary~\ref{cor:ex_AKE}, if $(C,u)$ is existentially closed in its completion $(K,v)=(\hat{C},\hat{u})$, then in particular $\hat{C}/C$ is separable. 
\end{remark}

As a consequence we also obtain the following variant of \cite[Theorem 7.1]{AF} and its corollaries:

\begin{corollary}\label{cor:notdiscrete}
Assume \eqref{R:ExClosed}.
Let $(K,v)$ and $(L,w)$ be equicharacteristic nontrivially valued henselian fields with the common trivially valued  subfield $C$.
We identify $C$ with its image under ${\rm res}_v$ and under ${\rm res}_w$ and assume that
$Kv/C$ and $Lw/C$ are separable.
If ${\rm Th}_\exists(Kv,C)\subseteq{\rm Th}_\exists(Lw,C)$,
then ${\rm Th}_\exists(K,v,C)\subseteq{\rm Th}_\exists(L,w,C)$.
\end{corollary}

\begin{proof}
Let $(K',v')$ be a henselian extension of $(K,v)$
that has a uniformizer $\pi$ (say a fresh indeterminate)
and satisfies
$Kv=K'v'$ -- 
for example, let $(K',\tilde{v})$ be a henselian extension of $(K,v)$ with $K'\tilde{v}=Kv(\!(t)\!)$ as in \cite[Lemma 2.1]{AF2}, 
and let $v'$ be the composed valuation $v_t\circ \tilde{v}$.
Then both $(K',v',\pi)$ and $(Lw(\pi)^{h},v_{\pi},\pi)$ are extensions of the
equicharacteristic $\mathbb{Z}$-valued field 
$(C(\pi),v_{\pi})$.
Moreover, 
the extension
$K'v'/C(\pi)v_{\pi}$
is equal to $Kv/C$,
and 
$(Lw(\pi)^{h})v_{\pi}/C(\pi)v_{\pi}$
is equal to $Lw/C$,
so we may apply Proposition~\ref{prop:main} to obtain
${\rm Th}_\exists(K',v',\pi,C(\pi)) \subseteq {\rm Th}_\exists(Lw(\pi)^h,v_{\pi},\pi,C(\pi))$.
Taking reducts yields ${\rm Th}_\exists(K,v,C) \subseteq {\rm Th}_\exists(K',v',C) \subseteq {\rm Th}_\exists(Lw(\pi)^h,v_{\pi},C)$.

By Lemma \ref{lem:section} there exists
an elementary extension $(L,w)\prec(L^*,w^*)$
with a partial section $\zeta \colon Lw\rightarrow L^*$ of ${\rm res}_{w^*}$ over $C$, which then extends
to an embedding $(Lw(\pi)^h,v_{\pi})\rightarrow(L^*,w^*)$ over $C$
by sending $\pi$ to any nonzero element
of positive value.
Hence we obtain the inclusion ${\rm Th}_\exists(Lw(\pi)^h,v_\pi,C) \subseteq {\rm Th}_\exists(L^*,w^*,C) = {\rm Th}_\exists(L,w,C)$.
\end{proof}

\begin{remark}\label{rem:R4}
We observe that the assumption
\eqref{R:ExClosed} 
is necessary in Corollary \ref{cor:notdiscrete}
(and hence in Proposition~\ref{prop:main} and also, less immediately, in Theorem \ref{thm:main}):
If $C$ is large and $K/C$ an extension 
with a valuation $v$ on $K/C$ with $Kv=C$,
we apply Corollary \ref{cor:notdiscrete} with
$(L,w)=(C(\!(t)\!),v_t)$;
since $Kv=C=Lw$, we conclude that ${\rm Th}_\exists(K,v,C)={\rm Th}_\exists(C(\!(t)\!),v_t,C)$,
so $C\prec_\exists C(\!(t)\!)$
implies that $C\prec_\exists K$,
hence \eqref{R:ExClosed} must hold.
\end{remark}

\begin{remark}
As noted in Remark \ref{rem:Rn},
the statement of \eqref{R:ExClosed} is known to hold at least for certain base fields $K$,
and inspection of our proofs shows that 
the assumption that \eqref{R:ExClosed} holds can be weakened 
to require only that it holds when restricted to specific base fields:
In Corollary \ref{cor:large} 
it suffices to assume that \eqref{R:ExClosed} holds when restricted to 
elementary extensions of $Kv$.
In Proposition \ref{prop:main}
it therefore suffices to assume that \eqref{R:ExClosed} holds when restricted to 
residue fields of proper coarsenings of elementary extensions of $(K,v)$ and $(L,w)$.
In particular, both here and similarly
in Theorem \ref{thm:main}, Corollary \ref{cor:ex_AKE} and Corollary \ref{cor:notdiscrete}
it suffices to assume that \eqref{R:ExClosed} holds when restricted to
extensions of $C$.
\end{remark}

\section*{Acknowledgements}
\noindent
A.~F.~was funded by the Deutsche Forschungsgemeinschaft (DFG) - 404427454.
S.~A.\ was supported by GeoMod AAPG2019 (ANR-DFG).
S.~A.\ and A.~F.\ were supported by the Institut Henri Poincaré.

\def\bibfont{\footnotesize}
\bibliographystyle{plain}

\end{document}